\documentclass[runningheads,a4paper]{llncs}

\usepackage{amssymb}
\usepackage{graphicx}

\usepackage{url}
\urldef{\mailsa}\path|{alfred.hofmann, ursula.barth, ingrid.haas, frank.holzwarth,|
\urldef{\mailsb}\path|anna.kramer, leonie.kunz, christine.reiss, nicole.sator,|
\urldef{\mailsc}\path|erika.siebert-cole, peter.strasser, lncs}@springer.com|

\usepackage{amsfonts}

\newcommand{\CC}{{\mathbb{C}}}
\newcommand{\NN}{{\mathbb{N}}}
\newcommand{\RR}{{\mathbb{R}}}
\newcommand{\TT}{{\mathbb{T}}}
\newcommand{\ZZ}{{\mathbb{Z}}}
\newcommand{\cD}{{\mathcal{D}}}
\newcommand{\cS}{{\mathcal{S}}}
\newcommand{\cP}{{\mathcal{P}}}
\newcommand{\cQ}{{\mathcal{Q}}}

\newcommand{\Span}{{\mathop{\rm span\,}}}
\newcommand{\diag}{{\mathop{\rm diag}}}
\newcommand{\rank}{{\mathop{\rm rank\,}}}

\newcommand{\Stirl}[2]{{\left\{ #1 \atop #2 \right\}}}
\newcommand{\Stirli}[2]{{\left[ #1 \atop #2 \right]}}

\begin{document}

\mainmatter  

\title{Reconstructing sparse exponential polynomials from samples:
  difference operators,
  Stirling numbers and Hermite interpolation}

\titlerunning{Reconstruction of sparse exponential polynomials}

%
%
\author{Tomas Sauer}
\authorrunning{Tomas Sauer}

\institute{Lehrstuhl f\"ur Mathematik mit Schwerpunkt Digitale
  Bildverarbeitung \& FORWISS, University of Passau, D--94032 Passau
  \\
  \url{Tomas.Sauer@uni-passau.de}
  \url{http://www.fim.uni-passau.de/digitale-bildverarbeitung}}

%
%

\toctitle{Reconstructing sparse exponential polynomials from samples}
\tocauthor{Tomas Sauer}
\maketitle

\begin{abstract}
  Prony's method, in its various concrete algorithmic realizations, is
  concerned with the reconstruction of a sparse exponential sum from
  integer samples. In several variables, the reconstruction is based
  on finding the variety for a zero dimensional radical ideal. If one
  replaces the coefficients in the representation by polynomials,
  i.e., tries to recover sparse exponential polynomials, the zeros
  associated to the ideal have multiplicities attached to them. The
  precise relationship
  between the coefficients in the exponential polynomial and the
  multiplicity spaces are pointed out in this paper.
\end{abstract}

\section{Introduction}
In this paper we consider an extension of what is known as
\emph{Prony's problem}, namely the reconstruction of a function
\begin{equation}
  \label{eq:exppolyDef}
  f(x) = \sum_{\omega \in \Omega} f_\omega (x) \, e^{\omega^T x},
  \qquad 0 \neq f_\omega \in \Pi, \qquad \omega \in \Omega \subset
  \left( \RR + i \TT \right)^s,
\end{equation}
from multiinteger samples, i.e., from samples $f(\Lambda)$ of $f$ on a
subgrid $\Lambda$
of $\ZZ^s$. Here, the function $f$ in (\ref{eq:exppolyDef}) is assumed to be
a sparse
\emph{exponential polynomial} and the original version of Prony's
problem, stated in one variable in \cite{prony95:_essai}, is the case where
all $f_\omega$ are \emph{constants}. Here ``sparsity''
refers to the fact that the cardinality of $\Omega$ is 
small and that the frequencies are either too unstructured or too
irregularly spread to be analyzed, for example, by means of Fourier
transforms.

Exponential polynomials appear quite frequently in various fields of
mathematics, for example they are known to be exactly the homogeneous
solutions of partial differential equations 
\cite{groebner37:_ueber_macaul_system_bedeut_theor_differ_koeff,groebner39:_ueber_eigen_integ_differ_koeff} 
or partial difference equations \cite{Sauer16:_kernels} with constant
coefficients. 

I learned about the generalized problem (\ref{eq:exppolyDef})
from a very interesting talk of Bernard Mourrain at the 2016 MAIA
conference, September 2016. Mourrain
\cite{Mourrain16P} studies extended and generalized Prony problems,
especially of the form (\ref{eq:exppolyDef}), but
also for $\log$-polynomials, by means of sophisticated
algebraic techniques like Gorenstein rings and truncated Hankel
operators. Much of it is based on the classical duality between series and
polynomials that was used in the definition of \emph{least interpolation}
\cite{deBoorRon92a}. He also gives a recovery algorithm based on
finding separating lines, a property that has to defined \emph{a
  posteriori} and that can lead to severe numerical problems.

This paper here approaches the problem in a different, more direct and
elementary way, following the
concepts proposed in
\cite{Sauer15S,Sauer2016A}, namely by using as a main tool the
factorization of a certain Hankel matrix in
terms of Vandermonde matrices; this factorization, stated in
Theorem~\ref{T:HankelFact} has the advantage to give a handy
criterion for sampling sets and was a useful tool for understanding
Prony's problem in several variables, cf. \cite{Sauer2016A}. Moreover,
the approach uses connections to the
description of finite dimensional kernels of multivariate convolutions
or, equivalently, the homogeneous solutions of systems of partial
difference operators.

Before we explore Prony's problem in detail, we show in
Section~\ref{sec:Kernels} that it can also
be formulated in terms of kernels of convolution operators or,
equivalently, in terms of homogeneous solutions of partial difference
equations. From that perspective it is not too surprising that many of
the tools used here are very similar to the ones from \cite{Sauer16:_kernels}.
The relationship that connects finite
differences, the Taylor expansion and the Newton form of interpolation
on the multiinteger grid, can be conveniently expressed in terms of
multivariate Stirling numbers of the second kind and will be
established in Section~\ref{sec:Stirling}. In
Section~\ref{sec:IdealHermite}, this background will be applied to define the
crucial ``Prony ideal'' by means of a generalized Hermite interpolation
problem that yields an \emph{ideal projector}. The aforementioned
factorization, stated and proved in Section~\ref{sec:Prony} then
allows us to directly extend the algorithms from
\cite{Sauer15S,Sauer2016A} which 
generate ideal bases and multiplication tables to the
generalized problem without further work. How the eigenvalues of the
multiplication tables relate to the common zeros of the ideal in the
presence of multiplicities is finally pointed out and discussed in
Section~\ref{sec:MultTab}.

The notation used in this paper is as follows. By $\Pi = \CC
[x_1,\dots,x_s]$ we denote the ring of polynomials with complex
coefficients. For $A \subset \NN_0^s$ we denote by $\Pi_A = \Span \{ 
(\cdot)^\alpha : \alpha \in A \} \subset \Pi$ the vector space spanned
by the monomials with exponents in $A$, using the fairly common
notation $(\cdot)^\alpha$ for the monomial (function) $m \in \Pi$,
defined as $m(x) = x^\alpha$.
The set of all multiindices
$\alpha \in \NN_0^s$ of \emph{length} $|\alpha| = \alpha_1 + \cdots +
\alpha_s$ is written as $\Gamma_n := \{ \alpha \in
\NN_0^s : |\alpha| \le n \}$ and defines $\Pi_n = \Pi_{\Gamma_n}$,
the vector space of polynomials of total degree at most $n$.

\section{Kernels of difference operators}
\label{sec:Kernels}
There is a different point of view for Prony's problem
(\ref{eq:exppolyDef}), namely in terms of \emph{difference operators}
and their kernels. Recall that a difference equation can be most
easily written as
\begin{equation}
  \label{eq:DiffEq}
  q(\tau) u = v, \qquad u,v : \ZZ^s \to \CC, \quad q \in \Pi,
\end{equation}
where $\tau$ stands for the \emph{shift operator} defined as
$\tau_j u := u (\cdot + \epsilon_j)$, $j=1,\dots,s$, and $\tau^\alpha
:= \tau_1^{\alpha_1} \cdots \tau_s^{\alpha_s}$, $\alpha \in \NN_0^s$.
Then the
difference equation (\ref{eq:DiffEq}) takes the explicit form
\begin{equation}
  \label{eq:ConvolDef}
  v = \left( \sum_{\alpha \in \NN_0^s} q_\alpha \tau^\alpha \right) f
  = \sum_{\alpha \in \NN_0^s} q_\alpha \tau^\alpha u = \sum_{\alpha \in
    \NN_0^s} q_\alpha \, u (\cdot + \alpha) = a * u  
\end{equation}
where $a$ is a finitely supported sequence of filter coefficients with
nonzero coefficients $a (-\alpha) = q_\alpha$. This is the well known
fact that any difference equation is equivalent to an FIR filter or a
convolution or correlation with a finite sequence.

Of particular interest are \emph{kernels} of the convolution operators
or, equivalently, homogeneous solutions for the difference equation
(\ref{eq:DiffEq}), or, more generally of a finite \emph{system} of difference
equations
\begin{equation}
  \label{eq:DiffEqSys}
  q (\tau) u = 0, \qquad q \in Q \subset \Pi, \, \# Q < \infty.
\end{equation}
Indeed, it is easily seen that (\ref{eq:DiffEqSys}) depends on the
space $\cQ = \Span Q$ and not on the individual basis. But even more
is true. Since, for any polynomials $g_q \in \Pi$, $q \in Q$, we also have
$$
\sum_{q \in Q} (g_q \, q) (\tau) u = \sum_{q \in Q} g_q (\tau) (
q(\tau) u ) = \sum_{q \in Q} g_q (\tau) \, 0 = 0,
$$
the space of homogeneous solutions of (\ref{eq:DiffEqSys}) depends on
the \emph{ideal} $\langle Q \rangle$ generated by $Q$ and any (ideal)
basis of this ideal defines a system of difference equations with the
same solutions.

The kernel space
$$
\ker Q(\tau) := \{ u : q(\tau) u = 0, \, q \in Q \},
$$
on the other hand, also has an obvious structure:
$$
0 = p(\tau) 0 = p(\tau) q(\tau) u = q(\tau) p(\tau) u, \qquad p \in \Pi,
$$
tells us that $u \in \ker Q(\tau)$ if and only if $p(\tau) \in \ker
Q(\tau)$ for any $p \in \Pi$ so that the homogeneous spaces are closed
under translation. This already indicates that relationships between
translation invariant polynomial spaces and ideals will play a crucial
role.

In one variable, the homogeneous solutions of difference equations are
known to be exactly the exponential polynomial sequence, i.e., the
sequences obtained by sampling exponential polynomials of the form
(\ref{eq:exppolyDef}), see \cite{jordan65:_calcul}. In several
variables, the situation is slightly more intricate and studied in
\cite{Sauer16:_kernels} where a characterization is given for the case
that $\ker Q (\tau)$ is \emph{finite dimensional}. Indeed, the kernel
spaces are of the form
$$
\bigoplus_{\omega \in \Omega} \cP_\omega \, e^{\omega^T \cdot},
$$
where $\Omega \subset \CC^s$ is a finite set of frequencies and
$\cP_\omega \subset \Pi$ is a finite dimensional translation invariant
subspace of polynomials.

Therefore, we can reformulate the problem of reconstructing a function
of the form (\ref{eq:exppolyDef}) from integer samples, i.e., from
\begin{equation}
  \label{eq:fsequence}
  \alpha \mapsto f(\alpha) = \sum_{\omega \in \Omega} f_\omega (\alpha) \,
  e^{\omega^T \alpha}, \qquad \alpha \in \ZZ^s,  
\end{equation}
as the problem of finding, for this exponential polynomial sequence, a
system $Q$ of partial difference equations such that $Q(\tau) f =
0$. This finding of homogenizing equations is clearly the dual of
finding homogeneous solutions of a given system. In fact, we will also
study the question of how many elements of the sequence
(\ref{eq:fsequence}) we have to know in order to generate the ideal
and how we can finally recover $f$ again from the dual equations. In
this respect, we can reformulate the construction from the subsequent
sections in the following form.

\begin{theorem}
  Given any exponential polynomial sequence $f : \ZZ^s \to \CC$ of the
  form (\ref{eq:fsequence}), there exists a finite set $Q \subset \Pi$
  of polynomials, the so--called \emph{Prony ideal} of $f$, such that 
  $$
  \Span \{ \tau^\alpha f : \alpha \in \ZZ^s \} = \ker Q(\tau),
  $$
  and the set $Q$ can be constructed from finitely many values of $f$.
\end{theorem}

\section{Stirling numbers and invariant spaces of polynomials}
\label{sec:Stirling}
The classical \emph{Stirling numbers} of the second kind, written as
$\Stirl{n}{k}$ in Karamata's notation,
cf. \cite[p.~257ff]{GrahamKnuthPatashnik98}, can be defined as
\begin{equation}
  \label{eq:Stirl1d}
  \Stirl{n}{k} := \frac1{k!} \sum_{j=0}^k (-1)^{k-j} {k \choose j} \, j^n.
\end{equation}
One important property is that they are \emph{differences of zero}
\cite{gould71:_noch_stirl_zahlen}, which means that 
$$
\Stirl{n}{k} = \frac1{k!} \Delta^k 0^n := \frac1{k!} \left( \Delta^k (\cdot)^n
\right) (0).
$$
Since this will turn out to be a very useful property, we define the
multivariate Stirling numbers of the second kind for $\nu,\kappa \in
\ZZ^s$ as
\begin{equation}
  \label{eq:StirlingMulti}
  \Stirl{\nu}{\kappa} := \frac1{\kappa!} \left( \Delta^\kappa (\cdot)^\nu
  \right) (0) = \frac1{\kappa!} \sum_{\gamma \le \kappa}
  (-1)^{|\kappa| - |\gamma|} {\kappa \choose \gamma} \, \gamma^\nu,
\end{equation}
with the convention that $\Stirl{\nu}{\kappa} = 0$ if $\kappa \not\le
\nu$ where $\alpha \le \beta$ if $\alpha_j \le \beta_j$,
$j=1,\dots,s$. Moreover, we use the usual definition
$$
{\kappa \choose \gamma} := \prod_{j=1}^s {\kappa_j \choose \gamma_j}.
$$
The identity in (\ref{eq:StirlingMulti}) follows from
the definition of the difference operator
$$
\Delta^\kappa := ( \tau - I )^\kappa, \qquad \tau p := \left[
  p(\cdot+\epsilon_j) : j =1,\dots,s \right],  \quad p \in \Pi,
$$
by straightforward computations. From \cite{Sauer16:_kernels} we
recall the degree preserving operator
\begin{equation}
  \label{eq:LOperator}
  L p := \sum_{|\gamma| \le \deg p} \frac1{\gamma!} \Delta^\gamma p
  (0) \, (\cdot)^\gamma, \qquad p \in \Pi,
\end{equation}
which has a representation in terms of Stirling numbers: if $p =
\sum_\alpha p_\alpha \, (\cdot)^\alpha$, then
$$
L p := \sum_{|\gamma| \le \deg p} (\cdot)^\gamma \,
\sum_{|\alpha| \le \deg p} p_\alpha \frac1{\gamma!} \left(
  \Delta^\gamma (\cdot)^\alpha \right) (0)
= \sum_{|\gamma| \le \deg p} \left( \sum_{|\alpha| \le \deg p}
  \Stirl{\alpha}{\gamma} \, p_\alpha \right) \, (\cdot)^\gamma,
$$
that is,
\begin{equation}
  \label{eq:StirlingMultiplier}
  ( Lp )_\alpha = \sum_{\beta \in \NN_0^s}  \Stirl{\beta}{\alpha} \,
  p_\beta, \qquad \alpha \in \NN_0^s.
\end{equation}
With the \emph{Pochhammer symbols} or \emph{falling
  factorials}
\begin{equation}
  \label{eq:PochhammerDef}
  (\cdot)_\alpha := \prod_{j=1}^s \prod_{k=0}^{\alpha_j-1} \left(
    (\cdot)_j - k \right),
\end{equation}
the inverse of $L$ takes the form
\begin{equation}
  \label{eq:L-1Operator}
  L^{-1} p := \sum_{|\gamma| \le \deg p} \frac1{\gamma!} D^\gamma p
  (0) \, (\cdot)_\gamma,
\end{equation}
see again \cite{Sauer16:_kernels}. The \emph{Stirling
  numbers of first kind}
\begin{equation}
  \label{eq:Stirling1stDef}
  \Stirli{\nu}{\kappa} := \frac1{\kappa!} \left( D^\kappa (\cdot)_\nu
  \right) (0),
\end{equation}
allow us to express the inverse $L^{-1}$ in analogous way for the
coefficients of the representation $p =
\sum_\alpha \hat p_\alpha \, (\cdot)_\alpha$. Indeed,
$$
\left( L^{-1} p \right)^\wedge_\alpha = \sum_{\beta \in \NN_0^s}
\Stirli{\beta}{\alpha} \, \hat p_\beta.
$$
By the Newton
interpolation formula for integer sites,
cf. \cite{IsaacsonKeller66,Steffensen27}, and the Taylor formula we
then get
\begin{eqnarray*}
  (\cdot)^\alpha
  & = & \sum_{\beta \le \alpha} \frac1{\beta!} \left( \Delta^\beta
        (\cdot)^\alpha \right) (0) \, (\cdot)_\beta
        = \sum_{\beta \in \NN_0^s}
        \Stirl{\alpha}{\beta} \, (\cdot)_\beta \\
  & = & \sum_{\beta \in \NN_0^s}
        \Stirl{\alpha}{\beta} \, \sum_{\gamma \le 
        \beta} \frac1{\gamma!} \left( D^\gamma (\cdot)_\beta \right)
        (0) \, (\cdot)^\gamma
        = \sum_{\beta,\gamma \in \NN_0^s} \Stirl{\alpha}{\beta}
        \Stirli{\beta}{\gamma} \, (\cdot)^\gamma
\end{eqnarray*}
from which a comparison of coefficients yields the extension of the
well--known duality between the Stirling numbers of the two kinds to
the multivariate case:
\begin{equation}
  \label{eq:StirlDual}
  \sum_{\beta \in \NN_0^s} \Stirl{\alpha}{\beta} \,
  \Stirli{\beta}{\gamma}
  = \delta_{\alpha,\gamma}, \qquad \alpha,\beta \in \NN_0^s.
\end{equation}
Moreover, the multivariate Stirling numbers satisfy a recurrence
similar to the univariate case. To that end, note that the Leibniz
rule for the forward difference, cf. \cite{deBoor05}, yields
$$
\Delta^\kappa (\cdot)^{\nu+\epsilon_j} = \Delta^\kappa \left(
  (\cdot)^\nu \, (\cdot)^{\epsilon_j} \right)
= \kappa_j \, \Delta^\kappa (\cdot)^\nu + \Delta^{\kappa-\epsilon_j} \,
(\cdot)^\nu,
$$
which we substitute into (\ref{eq:StirlingMulti}) to obtain the recurrence
\begin{equation}
  \label{eq:Stirl2ndRek}
  \Stirl{\nu+\epsilon_j}{\kappa}
  = \frac1{\kappa!} \left( \kappa_j \, \Delta^\kappa (\cdot)^\nu +
    \Delta^{\kappa-\epsilon_j}  (\cdot)^\nu \right) (0)
  = \kappa_j \Stirl{\nu}{\kappa} + \Stirl{\nu}{\kappa-\epsilon_j}.
\end{equation}
The operator $L$ also can be used to relate structures between
polynomial subspaces.

\begin{remark}
  Except \cite{schreiber15:_multiv_sterl}, which however does not
  connect to the above, I was not able to find references about
  \emph{multivariate} Stirling numbers, so the above simple and elementary
  proofs are added for the sake of completeness. Nevertheless, Gould's
  statement from \cite{gould71:_noch_stirl_zahlen} may well be true:
  ``\textsl{\dots aber es mag von Interesse sein, da{\ss} mindestestens
    tausend Abhandlungen in der Literatur existieren, die sich mit den
    Stirlingschen Zahlen besch\"aftigen. Es ist also sehr schwer,
    etwas Neues \"uber die Stirlingschen Zahlen zu entdecken.}''
\end{remark}

\begin{definition}
  A subspace ${\mathcal P}$ of $\Pi$ is called \emph{shift invariant}
  if
  \begin{equation}
    \label{eq:ShiftInvarDef}
    p \in {\mathcal P} \qquad \Leftrightarrow \qquad
    p ( \cdot + \alpha ) \in {\mathcal P}, \quad \alpha \in \NN_0^s,
  \end{equation}
  and it is called $D$--invariant if
  \begin{equation}
    \label{eq:DInvarDef}
    p \in {\mathcal P} \qquad \Leftrightarrow \qquad
    D^\alpha p \in {\mathcal P}, \quad \alpha \in \NN_0^s,
  \end{equation}
  where $D^\alpha = \frac{\partial^{|\alpha|}}{\partial x^\alpha}$.
  The \emph{principal shift-} and \emph{$D$--invariant} spaces for a
  polynomial $p \in \Pi$ are defined as
  \begin{equation}
    \label{eq:PrincInvar}
    {\mathcal S} (p) := \Span \{ p(\cdot+\alpha) : \alpha \in \NN_0^s
    \}, \qquad
    {\mathcal D} (p) := \Span \{ D^\alpha p : \alpha \in \NN_0^s
    \},
  \end{equation}
  respectively.
\end{definition}

\begin{proposition}\label{P:LShiftD}
  A subspace ${\mathcal P}$ of $\Pi$ is \emph{shift invariant} if and
  only if $L {\mathcal P}$ is $D$--invariant.
\end{proposition}

\begin{proof}
  The direction ``$\Leftarrow$'' has been shown in
  \cite[Lemma~3]{Sauer16:_kernels}, so assume that ${\mathcal P}$ is shift
  invariant and consider, for some $\alpha \in \NN_0^s$,
  \begin{eqnarray*}
    D^\alpha L p
    & = & D^\alpha \sum_{|\gamma| \le \deg p} \frac1{\gamma!} \Delta^\gamma p
          (0) \, (\cdot)^\gamma \\
    & = & \sum_{\gamma \ge \alpha} \frac1{(\gamma-\alpha)!}
          \Delta^{\gamma-\alpha} \left( \Delta^\alpha p \right) (0) \,
          ( \cdot)^{\gamma-\alpha}
          = L \Delta^\alpha p,
  \end{eqnarray*}
  where $\Delta^\alpha p \in {\mathcal P}$ since the space is shift
  invariant. Hence $D^\alpha Lp \in L {\mathcal P}$ which proves that
  this space is indeed $D$--invariant.
  \qed
\end{proof}

\noindent
A simple and well--known consequence of Proposition~\ref{P:LShiftD}
can be recorded as follows.

\begin{corollary}
  A subspace ${\mathcal P}$ of $\Pi$ is invariant under integer shifts
  if and only if it is invariant under arbitrary shifts.
\end{corollary}

\begin{proof}
  If together with $p$ also all $p(\cdot+\alpha)$ belong to $\mathcal
  P$ then, by Proposition~\ref{P:LShiftD}, the space $L {\mathcal P}$
  is $D$--invariant from which it follows by
  \cite[Lemma~3]{Sauer16:_kernels} that $p \in \mathcal P = L^{-1} L
  {\mathcal P}$ implies $p(\cdot + y) \in \mathcal P$, $y \in \CC^s$.
\end{proof}

\begin{proposition}\label{P:LSDL}
  For $q \in \Pi$ we have that $L \cS( q ) = \cD ( Lq )$.
\end{proposition}

\begin{proof}
  By Proposition~\ref{P:LShiftD}, $L \cS (q)$ is a $D$--invariant space
  that contains $Lq$, hence $L \cS (q) \supseteq \cD ( Lq )$. On the
  other hand $L^{-1} \cD ( Lq )$ is a shift invariant space containing
  $Lq$, hence
  $$
  L^{-1} \cD ( Lq ) \supseteq \cS (L^{-1} L q ) = \cS (q),
  $$
  and applying the invertible operator $L$ to both sides of the
  inclusion yields that $L \cS (q) \subseteq \cD ( Lq )$ and completes
  the proof.
  \qed
\end{proof}

\noindent
Stirling numbers do not only relate invariant spaces, they also are
useful for studying another popular differential operator. To that end, we
define the partial differential operators
\begin{equation}
  \label{eq:thetaOpDef}
  \frac{\hat \partial}{\hat \partial x_j} = (\cdot)_j
  \frac{\partial}{\partial x_j}
  \qquad \mbox{and} \qquad
  \hat D^\alpha := \frac{\hat \partial^\alpha}{\hat \partial
    x^\alpha}, \qquad \alpha \in \NN_0^s,
\end{equation}
also known as \emph{$\theta$--operator} in the univariate case. Recall
that the multivariate $\theta$--operator is usually of the form
$$
\sum_{|\alpha| = n} \hat D^\alpha
$$
and its eigenfunctions are the homogeneous polynomials, the associated
eigenvalues is their total degree. Here, however, we need the partial
$\theta$--operators. To
relate differential operators based on $\hat D$ to standard
differential operators, we use the notation $( \xi D)^\alpha :=
\xi^\alpha D^\alpha$ for the \emph{$\xi$ scaled} partial derivatives,
$\xi \in \CC^s$ and use, as common, $\CC_* := \CC \setminus \{ 0 \}$.

\begin{theorem}\label{T:thetaOpDiffOp}
  For any $q \in \Pi$ and $\xi \in \CC^s$ we have that
  \begin{equation}
    \label{eq:thetaOpDiffOp}
    \left( q(\hat D) \right) p (\xi) = \left( Lq ( \xi D ) \right) p
    (\xi), \qquad p \in \Pi.
  \end{equation}
\end{theorem}

\begin{proof}
  We prove by induction that 
  \begin{equation}
    \label{eq:TthetaOpDiffOpPf1}
    \hat D^\alpha = \sum_{\beta \le \alpha} \Stirl{\alpha}{\beta} \,
    (\cdot)^\beta \, D^\beta, \qquad \alpha \in \NN_0^s,
  \end{equation}
  which is trivial for $\alpha = 0$. The inductive step uses the
  Leibniz rule to show that
  \begin{eqnarray*}
    \hat D^{\alpha + \epsilon_j}
    & = & x_j \frac{\partial}{\partial x_j} \, \sum_{\beta \le \alpha}
    \Stirl{\alpha}{\beta} \, (\cdot)^\beta \, D^\beta
    = \sum_{\beta \le \alpha}
    \Stirl{\alpha}{\beta} \, \left( \beta_j (\cdot)^\beta \, D^\beta +
      (\cdot)^{\beta+\epsilon_j} \, D^{\beta+\epsilon_j} \right) \\
    & = & \sum_{\beta \le \alpha + \epsilon_j} \left( \beta_j \,
      \Stirl{\alpha}{\beta} + \Stirl{\alpha}{\beta - \epsilon_j}
    \right) \, (\cdot)^\beta \, D^\beta,
  \end{eqnarray*}
  from which (\ref{eq:TthetaOpDiffOpPf1}) follows by taking into
  account the recurrence (\ref{eq:Stirl2ndRek}). Thus, by
  (\ref{eq:StirlingMultiplier}),
  $$
  q(\hat D) = \sum_{\alpha \in \NN_0^s} q_\alpha \, \hat D^\alpha
  = \sum_{\alpha \in \NN_0^s} q_\alpha \, \sum_{\beta \in \NN_0^s}
  \Stirl{\alpha}{\beta} \, (\cdot)^\beta \, D^\beta
  = \sum_{\beta \in \NN_0^s} ( Lq )_\beta \, (\cdot)^\beta \, D^\beta,
  $$
  and by applying the differential operator to $p$ and evaluating at
  $\xi$, we obtain (\ref{eq:thetaOpDiffOp}).
  \qed
\end{proof}

\section{Ideals and Hermite interpolation}
\label{sec:IdealHermite}

A set $I \subseteq \Pi$ of polynomials is called an \emph{ideal} in
$\Pi$ if it is closed under addition and multiplication with arbitrary
polynomials. A projection $P : \Pi \to \Pi$ is called an \emph{ideal
  projector}, cf. \cite{boor05:_ideal}, if $\ker P := \{ p \in \Pi :
Pp = 0
\}$ is an ideal. Ideal projectors with finite range are \emph{Hermite
  interpolants}, that is, projections $H : \Pi \to \Pi$ such that
\begin{equation}
  \label{eq:HermiteInterpolant}
  \left( q(D) Hp \right) (\xi) = q(D) p (\xi), \qquad q \in \cQ_\xi,
  \quad \xi \in \Xi,
\end{equation}
where $\cQ_\xi$ is a finite dimensional $D$--invariant subspace of
$\Pi$ and $\Xi \subset \CC^s$ is a finite set, cf
\cite{MarinariMoellerMora96}. A polynomial $p \in
\ker H$ vanishes at $\xi \in \Xi$ with \emph{multiplicity} $\cQ_\xi$,
see \cite{groebner37:_ueber_macaul_system_bedeut_theor_differ_koeff}
for a definition of multiplicity of the common zero of a set of
polynomials as a structured quantity.

A particular case is that $\cQ_\xi$ is a \emph{principal
  $D$--invariant space} of the form $\cQ_\xi = \cD
(q_\xi)$ for some $q_\xi \in \Pi$, i.e., the multiplicities are
generated by a single polynomial. We say that the respective Hermite
interpolation problem and the associated ideal are of \emph{principal
  multiplicity} in this case. By means of the differential operator
$\hat D$ these ideals are also created by shift invariant spaces.

\begin{theorem}\label{T:DhatIdeal}
  For a finite $\Xi \subset \CC^s$ and polynomials $q_\xi \in \Pi$,
  $\xi \in \Xi$, the polynomial space
  \begin{equation}
    \label{eq:DhatIdeal}
    \left\{ p \in \Pi : q(\hat D) p (\xi) = 0, q \in \cS ( q_\xi ),
      \xi \in \Xi \right\}
  \end{equation}
  is an ideal of principal multiplicity. Conversely, if $\xi \in
  \CC_*^s$ then any ideal of principal multiplicity can be written in
  the form (\ref{eq:DhatIdeal}).
\end{theorem}

\begin{proof}
  For $\xi \in \Xi$ we set $\cQ_\xi' = \cD ( L q_\xi )$ which equals
  $L \cS (q_\xi)$ by Proposition~\ref{P:LSDL}. Then, also
  $$
  \cQ_\xi := \left\{ q ( \diag \, \xi \; \cdot ) : q \in \cQ_\xi' \right\}
  $$
  is a $D$--invariant space generated by $L q_\xi ( \diag \, \xi \; \cdot
  )$, and by Theorem~\ref{T:thetaOpDiffOp} it follows that
  \begin{equation}
    \label{eq:DhatCondhat}
    q(\hat D) p(\xi) = 0, \quad q \in \cS ( q_\xi )  
  \end{equation}
  if and only if
  \begin{equation}
    \label{eq:DhatCondNorm}
    q(D) p(\xi) = 0, \quad q \in \cQ_\xi = \cD \left( L q_\xi ( \diag \,
      \xi \; \cdot) \right).    
  \end{equation}
  This proves the first claim, the second one follows from the
  observation that the process is reversible provided that $\diag \, \xi$
  is invertible which happens if and only if $\xi \in \CC_*^s$.
  \qed
\end{proof}

\noindent
The equivalence of (\ref{eq:DhatCondhat}) and (\ref{eq:DhatCondNorm})
shows that Hermite interpolations can equivalently formulated either
in terms of regular differential operators and differentiation
invariant spaces or in terms of $\theta$--operators and shift
invariant spaces.

The \emph{Hermite interpolation problem} based on $\Xi$ and
polynomials $q_\xi$ can now be phrased as follows: given $g \in \Pi$
find a polynomial $p$ (in some prescribed space) such that
\begin{equation}
  \label{eq:HermitInterpolDef}
  q (\hat D) p (\xi) = q (\hat D) g (\xi), \qquad q \in \cS( q_\xi ),
  \quad \xi \in \Xi.  
\end{equation}
Clearly, the number of interpolation conditions for this problem is
the \emph{total multiplicity}
$$
N = \sum_{\xi \in \Xi} \dim \cS (q_\xi).
$$
The name ``multiplicity'' ist justified here since $\dim \cQ_\xi$ is
the scalar multiplicity of a common zero of a set of polynomials and
$N$ counts the total multiplicity. Note however, that this information
is incomplete since problems with the same $N$ can nevertheless be
structurally different.

\begin{example}\label{Ex:Count}
  Consider $q_\xi (x) = x_1 x_2$ and $q_\xi (x) = x_1^3$. In both
  cases $\dim \cS (q_\xi) = 4$ although, of course, the spaces $\Span
  \{ 1,x_1,x_2,x_1 x_2 \}$ and $\Span \{ 1,x_1,x_1^2,x_1^3 \}$
  do not coincide.
\end{example}

A subspace $\mathcal P \subset \Pi$ of polynomials is called an
\emph{interpolation space} if for any $g \in \Pi$ there exists $p \in
\cP$ such that (\ref{eq:HermitInterpolDef}) is satisfied. A subspace
$\cP$ is called a \emph{universal interpolation space} of order $N$ if this is
possible for \emph{any} choice of $\Xi$ and $q_\xi$ such that
$$
\sum_{\xi \in \Xi} \dim \cS (q_\xi) \le N.
$$
Using the definition
$$
\Upsilon_n := \left\{ \alpha \in \NN_0^s : \prod_{j=1}^s ( 1+\alpha_j
  ) \le n \right\}, \qquad n \in \NN,
$$
of the \emph{first hyperbolic orthant}, the positive part of the
\emph{hyperbolic cross}, we can give the following statement that
also tells us that the Hermite interpolation problem is always
solvable.

\begin{theorem}\label{T:UniversalSapce}
  $\Pi_{\Upsilon_N}$ is a universal interpolation space for the
  interpolation problem (\ref{eq:HermitInterpolDef}).
\end{theorem}

\begin{proof}
  Since the interpolant to (\ref{eq:HermitInterpolDef}) is an ideal
  projector by Theorem~\ref{T:DhatIdeal}, its kernel, the set of all
  homogeneous solutions to (\ref{eq:HermitInterpolDef}), forms a zero
  dimensional ideal in $\Pi$. This ideal has a Gr\"obner basis, for
  example with respect to the graded lexicographical ordering, cf.
  \cite{CoxLittleOShea92}, and the remainders of division by this
  basis form the space $\Pi_A$ for some lower set $A \subset \NN_0^s$
  of cardinality $N$. Since
  $\Upsilon_N$ is the union of all lower sets of cardinality $\le N$,
  it contains $\Pi_A$ and therefore $\Pi_{\Upsilon_N}$ is a universal
  interpolation space.\qed
\end{proof}

\section{Application to the generalized Prony problem}
\label{sec:Prony}

We now use the tools of the preceding sections to investigate the
structure of the generalized Prony problem~(\ref{eq:exppolyDef}) and
to show how to
reconstruct $\Omega$ and the polynomials $f_\omega$ from integer
samples. As in \cite{Sauer15S,Sauer2016A} we start by
considering for $A,B \subset \NN_0^s$ the \emph{Hankel matrix}
\begin{equation}
  \label{eq:FABDef}
  F_{A,B} = \left[ f(\alpha+\beta) :
    \begin{array}{c}
      \alpha \in A \\ \beta \in B
    \end{array}
  \right]
\end{equation}
of samples.

\begin{remark}
  Instead of the Hankel matrix $F_{A,B}$ one might also consider the
  \emph{Toeplitz matrix}
  \begin{equation}
    \label{eq:ToeplitzMat}
    T_{A,B} = \left[ f(\alpha-\beta) :
      \begin{array}{c}
        \alpha \in A \\ \beta \in B
      \end{array}
    \right], \qquad A,B \subset \NN_0^s,  
  \end{equation}
  which would lead to essentially the same results. The main
  difference is the set on which $f$ is sampled, especially if $A,B$
  are chosen as the total degree index sets $\Gamma_n := \{ \alpha \in
  \NN_0^s : |\alpha| \le n \}$ for some $n \in \NN$.
\end{remark}

\begin{remark}
  For a coefficient vector $p = \left( p_\alpha : \alpha \in A \right)
  \in \CC^B$, the result of $F_{A,B} p$ is exactly the restriction of
  the convolution $a * f$ from (\ref{eq:ConvolDef}) with $a (-\alpha)
  = p(\alpha)$, $\alpha \in A$. With the Toeplitz matrix we get the
  even more direct $T_{A,B} p = (f * p) (A)$.
\end{remark}

\noindent
Given a finite set $\Theta \subset \Pi'$ of linearly independent
linear functionals on $\Pi$ and $A \subset \NN_0^s$ the monomial
\emph{Vandermonde matrix} for the interpolation problem at $\Theta$ is
defined as
\begin{equation}
  \label{eq:VandermondeDef}
  V( \Theta, A ) := \left[ \theta (\cdot)^\alpha :
    \begin{array}{c}
      \theta \in \Theta \\ \alpha \in A
    \end{array}
  \right].
\end{equation}
It is standard linear algebra to show that the interpolation problem
\begin{equation}
  \label{eq:ThetaInterpol}
  \Theta p = y, \quad y \in \CC^\Theta, \qquad \mbox{i.e.} \qquad
  \theta p = y_\theta, \quad \theta \in \Theta,
\end{equation}
has a solution for any data $y = \CC^\Theta$ iff $\rank V( \Theta,A )
\ge \# \Theta$ and that the solution is unique iff $V( \Theta,A )$ is
a nonsingular, hence square, matrix.

For our particular application, we choose $\Theta$ in the following way:
Let $Q_\omega$ be a basis for for the space $\cS ( f_\omega )$  and
set
$$
\Theta_\Omega := \bigcup_{\omega \in \Omega} \{ \theta_\omega
\, q(\hat D) : q \in Q_\omega \}, \qquad \theta_\omega p := p(
e^\omega).
$$
Since, for any $\omega \in \Omega$,
$$
\left( \Delta^\alpha f_\omega : |\alpha| = \deg f_\omega \right)
$$
is a nonzero vector of complex numbers or constant polynomials, we
know that $1 \in \cS ( f_\omega )$ and will therefore always make the
assumption that $1 \in Q_\omega$, $\omega \in \Omega$, which
corresponds to $\theta_\omega \in \Theta_\Omega$, $\omega \in
\Omega$. Moreover, we request without loss of generality that
$f_\omega \in Q_\omega$.

We pattern the Vandermonde matrix conveniently as
$$
V ( \Theta_\Omega, A ) = \left[ \left( q (\hat D)
    (\cdot)^\alpha \right) (e^\omega) :
  \begin{array}{c}
    q \in Q_\omega, \, \omega \in \Omega \\
    \alpha \in A
  \end{array}
\right]
$$
to obtain the following fundamental factorization of the Hankel matrix.

\begin{theorem}\label{T:HankelFact}
  The Hankel matrix $F_{A,B}$ can be factored into
  \begin{equation}
    \label{eq:HankelFact}
    F_{A,B} = V ( \Theta_\Omega, A )^T \, F \, V ( \Theta_\Omega, B ),
  \end{equation}
  where $F$ is a nonsingular block diagonal matrix independent of $A$
  and $B$.
\end{theorem}

\begin{proof}
  We begin with an idea by Gr\"obner
  \cite{groebner37:_ueber_macaul_system_bedeut_theor_differ_koeff},
  see also \cite{Sauer16:_kernels}, and first
  note that any $g \in Q_\omega$ can be written as
  $$
  g(x+y) = \sum_{q \in Q_\omega} c_q (y) \, q(x), \qquad c_q : \CC^s
  \to \CC.
  $$
  Since $g(x+y)$ also belongs to $\Span Q_\omega$ as a function in
  $y$ for fixed $x$, we conclude that $c_q (y)$ can also be written in
  terms of $Q_\omega$ and thus have obtained the \emph{linearization formula}
  \begin{equation}
    \label{eq:LinearizationForm}
    g (x+y) = \sum_{q,q' \in Q_\omega} a_{q,q'} (g) \, q(x) \, q' (y),
    \qquad a_{q,q'} (g) \in \CC,
  \end{equation}
  from
  \cite{groebner37:_ueber_macaul_system_bedeut_theor_differ_koeff}.
  Now consider
  \begin{eqnarray*}
    \lefteqn{ (F_{A,B})_{\alpha,\beta}
      = f(\alpha+\beta) = \sum_{\omega \in \Omega} f_\omega
      (\alpha+\beta) \, e^{\omega^T (\alpha + \beta)} } \\
    & = & \sum_{\omega \in \Omega} \sum_{q,q' \in Q_\omega} a_{q,q'} (
    f_\omega ) \, q(\alpha) \, e^{\omega^T \alpha} \, q'(\beta) \,
    e^{\omega^T \beta} \\
    & = & \sum_{\omega \in \Omega} \sum_{q,q' \in Q_\omega} a_{q,q'} (
    f_\omega ) \, \left( q(\hat D) (\cdot)^\alpha \right) (e^\omega)
    \, \left( q'(\hat D) (\cdot)^\beta \right) (e^\omega)  \\
    & = & \left( V(\Theta_\Omega,A) \, \diag \left( \left[ a_{q,q'} (
          f_\omega ) :
          \begin{array}{c}
            q \in Q_\omega \\ q' \in Q_\omega 
          \end{array}
        \right] : \omega \in \Omega \right) \, V(\Theta_\Omega,B)^T
    \right)_{\alpha,\beta},
  \end{eqnarray*}
  which already yields (\ref{eq:HankelFact}) with
  $$
  F := \diag \left( \left[ a_{q,q'} (
      f_\omega ) :
      \begin{array}{c}
        q \in Q_\omega \\ q' \in Q_\omega 
      \end{array}
    \right] : \omega \in \Omega \right) = \diag \left( A_\omega :
    \omega \in \Omega \right).
  $$
  It remains to prove that the blocks $A_\omega$
  of the block diagonal matrix $F$ are nonsingular. To that end, we recall that
  $f_\omega \in Q_\omega$, hence, by (\ref{eq:LinearizationForm}),
  $$
  f_\omega = f_\omega (\cdot + 0) = \sum_{q,q' \in Q_\omega} a_{q,q'}
  (f_\omega) \, q(x) \, q'(0),
  $$
  that is, by linear independence of the elements of $Q_\omega$,
  $$
  \sum_{q' \in Q_\omega} a_{q,q'} (f_\omega) \, q'(0) =
  \delta_{q,f_\omega}, \qquad q \in Q_\omega,
  $$
  which can be written as $A_\omega Q_\omega (0) = e_{f_\omega}$ where
  $Q_\omega$ also stands for the polynomial vector formed by the basis
  elements. Since $Q_\omega$ is a basis for $\cS( f_\omega )$, there
  exist finitely supported sequences $c_q : \NN_0^s \to \CC$, $q \in
  Q_\omega$, such that
  $$
  q = \sum_{\alpha \in \NN_0^s} c_q (\alpha) \, f( \cdot + \alpha),
  = \sum_{q',q'' \in Q_\omega} a_{q',q''} (f_\omega) \left(
    \sum_{\alpha \in \NN_0^s} c_q (\alpha) \, q'' (\alpha) \right) \, q'
  $$
  from which a comparison of coefficients allows us to conclude that
  $$
  A_\omega \, \sum_{\alpha \in \NN_0^s} c_q (\alpha) Q_\omega (\alpha)
  = e_q, \qquad q \in Q_\omega,
  $$
  which even gives an ``explicit'' formula for the columns of
  $A_\omega^{-1}$. \qed
\end{proof}

\begin{remark}
  A similar factorization of the Hankel matrix in terms of Vandermonde
  matrices for slightly different but equivalent Hermite problems
  has also been given in
  \cite[Proposition~3.18]{Mourrain16P}.
  However, the invertibility of the ``inner matrix'' $F$ was concluded
  there from the invertibility of the Hankel matrix and the assumption
  that $\Pi_A$ must be an \emph{interpolation space}, giving unique
  interpolants for the Hermite problem.
  Theorem~\ref{T:HankelFact}, on
  the other hand, does not need these assumptions, shows that $F$ is
  \emph{always} nonsingular and therefore extends
  the one given in \cite[(5)]{Sauer2016A} in a natural way.
\end{remark}

\begin{remark}
  If $F_{A,B}$ is replaced by the Toeplitz matrix from
  (\ref{eq:ToeplitzMat}), then the factorization becomes
   \begin{equation}
    \label{eq:HankelFact}
    T_{A,B} = W ( \Theta_\Omega, A ) \, F \, W ( \Theta_\Omega, B )^*,
    \qquad W ( \Theta_\Omega, A ) := V ( \Theta_\Omega, A )^T
  \end{equation}
  which has more similarity to a block Schur decomposition since now a
  Hermitian of the factorizing matrix appears.
\end{remark}

\noindent
Once the factorization (\ref{eq:HankelFact}) is established, the
results from \cite{Sauer15S,Sauer2016A} can be applied
literally and extend to the case of exponential polynomial
reconstruction directly. In particular, the following observation is
relevant for the termination of the algorithms. It says that if the
row index set $A$ is ``sufficiently rich'', then the full information
about the ideal $I_\Omega := \ker \Theta_\Omega$ can be extracted from
the Hankel matrix $F_{A,B}$.

\begin{theorem}\label{T:HilbFun}
  If $\Pi_A$ is an interpolation space for $\Theta_\Omega$, for
  example if $A = \Upsilon_N$, then
  \begin{enumerate}
  \item\label{it:THilbFun1} the function $f$ can be reconstructed from
    samples $f(A+B)$, 
    $A,B \subset \NN_0^s$, if and only if $\Pi_A$ and $\Pi_B$ are
    interpolation spaces for $\Theta_\Omega$.
  \item\label{it:THilbFun2} a vector $p \in \CC^B \setminus \{ 0 \}$ satisfies
    $$
    F_{A,B} p = 0 \qquad \Leftrightarrow \qquad
    \sum_{\beta \in B} p_\beta \, (\cdot)^\beta \in I_\Omega \cap
    \Pi_B.
    $$
  \item\label{it:THilbFun3} the mapping $n \mapsto \rank F_{A,\Gamma_n}$
    is the affine Hilbert function for the ideal $I_\Omega$.
  \end{enumerate}
\end{theorem}

\begin{proof}
  Theorem~\ref{T:HilbFun} is a direct consequence of
  Theorem~\ref{T:HankelFact} by means of elementary linear
  algebra. The proof of \ref{it:THilbFun1}) is a literal copy of that
  of \cite[Theorem~3]{Sauer2016A} for \ref{it:THilbFun2}) we note
  that, for $p \in \CC^B$,
  \begin{eqnarray*}
    F_{A,B} p & = & V ( \Theta_\Omega, A )^T \, F \, V ( \Theta_\Omega,
                    B ) p \\
    & = & V ( \Theta_\Omega, A )^T \, F \, \left[ \left( q (\hat D) p
    \right) (e^\omega) : q \in Q_\omega, \, \omega \in \Omega \right].    
  \end{eqnarray*}
  By assumption, $V ( \Theta_\Omega, A )^T$ has full rank, $F$ is
  invertible by Theorem~\ref{T:HankelFact}, and therefore $F_{A,B} p =
  0$ if and only if the polynomial $p$ belongs to $I_\Omega$. Finally,
  \ref{it:THilbFun3}) is an immediate consequence of~\ref{it:THilbFun2}).
\end{proof}

\noindent
Theorem~\ref{T:HilbFun} suggests the following generic algorithm: use
a nested sequence $B_0 \subset B_1 \subset B_2 \subset \cdots$ of
index sets in $\NN_0^s$ such that there exist $j(n) \in \NN$, $n \in
\NN$, such that $B_{j(n)} = \Gamma_n$. In other words: the subsets
progress in a \emph{graded} fashion. Then, for $j=0,1,\dots$
\begin{enumerate}
\item Consider the kernel of $F_{\Upsilon_N,B_j}$, these are the ideal
  elements in $\Pi_{B_j}$.
\item Consider the complement of the kernel, these are elements of the
  \emph{normal set} and eventually form a basis for an interpolation
  space.
\item Terminate if $\rank F_{\Upsilon_N,B_{j(n+1)}} = \rank
  F_{\Upsilon_N,B_{j(n)}}$ for some $n$.
\end{enumerate}
Observe that this task of computing an ideal basis from nullspaces of
matrices is \emph{exactly} the same as in Prony's problem with
constant coefficients. The difference lies only in the fact that now
the ideal is not \emph{radical} any more, but this is obviously
irrelevant for Theorem~\ref{T:HilbFun}.

Two concrete instances of this approach were presented and discussed
earlier: \cite{Sauer15S} uses $B_j = \Gamma_j$ and Sparse Homogeneous
Interpolation Techniques (DNSIN) to compute an orthonormal H--basis
and a graded basis for the ideal and the normal space,
respectively. Since these computations are based on \emph{orthogonal
  decompositions}, mainly $QR$ factorizations, it is numerically
stable and suitable 
for finite precision computations in a floating point environment.
A symbolic approach where the $B_j$ are generated by adding
multiindices according to a graded term order, thus using Sparse
Monomial Interpolation with Least Elements (SMILE), was introduced in
\cite{Sauer2016A}. This method is more efficient in terms of number of
computations and therefore suitable for a symbolic framework with exact
rational arithmetic.

\begin{remark}
  The only a priori knowledge these algorithms need to know is an
  upper estimate for the \emph{multiplicity} $N$.
\end{remark}

\noindent
It should be mentioned that also \cite{Mourrain16P} gives algorithms
to reconstruct frequencies and coefficients by first determining the
\emph{Prony ideal} $I_\Omega$; the way how these algorithms work and
how they are derived are different, however. It would be worthwhile to
study and understand the differences between and the advantages of the
methods.

While we will point out in the next section how the frequencies can be
determined by generalized eigenvalue methods, we still need to clarify
how the coefficients of the polynomials $f_\omega$ can be computed
once the ideal structure and the frequencies are determined. To that
end, we write
$$
f_\omega = \sum_{\alpha \in \NN_0^s} f_{\omega,\alpha} (\cdot)^\alpha
$$
and note that, with $\xi_\omega := e^\omega \in \CC_*^s$
$$
f(\beta) = \sum_{\omega \in \Omega} f_\omega (\beta) \, e^{\omega^T
  \beta}
= \sum_{\omega \in \Omega} \sum_{\alpha \in \NN_0^s} f_{\omega,\alpha}
\, \beta^\alpha \, \xi_\omega^\beta
= \sum_{\omega \in \Omega} \sum_{\alpha \in \NN_0^s} f_{\omega,\alpha}
\, \left( \hat D^\alpha (\cdot)^\beta \right) (\xi_\omega).
$$
In other words, we have for any choice of $A_\omega \subset \NN_0^s$,
$\omega \in \Omega$ and $B \subset \NN_0^s$ that
\begin{eqnarray*}
  f(B) & := & \left[ f(\beta) : \beta \in B \right] \\
  & = & \left[
    \left( \hat D^\alpha (\cdot)^\beta \right) (\xi_\omega) :
    \begin{array}{c}
      \beta \in B \\ \alpha \in A_\omega, \omega \in \Omega
    \end{array}
  \right] \left[
    f_{\omega,\alpha} : \alpha \in A_\omega, \omega \in \Omega
  \right]  \\
  & =: & G_{A,B} f_\Omega.
\end{eqnarray*}
The matrix $G_{A,B}$ is another Vandermonde matrix for a Hermite--type
interpolation problem with the functionals
\begin{equation}
  \label{eq:fCompHermProblem}
  \theta_\omega \hat D^\alpha, \qquad \alpha \in A_\omega, \, \omega \in \Omega.  
\end{equation}
The linear system
$$
G_{A,B} \, f_\Omega  = f(B)
$$
can thus be used to determine $f_\Omega$: first note that $L \Pi_n =
\Pi_n$ and therefore it follows by Theorem~\ref{T:DhatIdeal} that 
the interpolation problem is a Hermite problem, i.e., its kernel is an
ideal. If we set
$$
N = \sum_{\omega \in \Omega} { \deg f_\omega + s \choose s} - 1
$$
then, by Theorem~\ref{T:UniversalSapce}, the space $\Pi_{\Upsilon_N}$
is a universal interpolation space for the interpolation problem
(\ref{eq:fCompHermProblem}). Hence, with $A_\omega = \Gamma_{\deg
  f_\omega}$, the matrix $G_{A,\Upsilon_N}$ contains a nonsingular
square matrix of size $\# A \times \#A$ and the coefficient vector
$f_\Omega$ is the unique solution of the overdetermined interpolation
problem.

\begin{remark}
  The a priori information about the multiplicity $N$ of the
  interpolation points does not allow for an efficient reconstruction
  of the frequencies as it only says that there are at most $N$ points
  or points of local multiplicity up to $N$.
\end{remark}

\noindent
Nevertheless, the degrees $\deg f_\omega$, $\omega \in \Omega$, more
precisely, upper bounds for them, can be derived as a by-product of
the determination of the frequencies $\omega$ by means of
multiplication tables. To clarify this relationship, we briefly revise
the underlying theory, mostly due to M\"oller and Stetter
\cite{MoellerStetter95}, in the next section. 

\section{Multiplication tables and multiple zeros}
\label{sec:MultTab}

Having computed a good basis $H$ for the ideal $I_\Omega$ and a basis for
the normal set $\Pi / I_\Omega$, the final step consists of finding
the common zeros of $H$. The method of choice is still to use
eigenvalues of the multiplication tables,
cf. \cite{AuzingerStetter88,Stetter95}, but things become slightly
more intricate since we now have to consider the case of zeros with 
multiplicities, cf. \cite{MoellerStetter95}.

Let us briefly recall the setup in our particular case. The
multiplicity space at $\xi_\omega = e^\omega \in \CC_*^s$ is
$$
\cQ_\omega := \cD \left( L f_\omega ( \diag \, \xi_\omega \; \cdot ) \right)
$$
and since this is a $D$--invariant subspace, it has a graded basis
$Q_\omega$ where the highest degree element
in this basis can be chosen as $g_\omega := L f_\omega ( \diag \,
\xi_\omega \; \cdot)$. 
Since $\cQ_\omega = \cD (g_\omega)$, all other basis elements $q \in
Q_\omega$ can be written as $q = g_q (D) g_\omega$, $g_q \in \Pi$, $q
\in Q_\Omega$.

Given a basis $P$ of the normal set $\Pi / I_\Omega$ and a normal form operator
$\nu : \Pi \to \Pi / I_\Omega = \Span P$ modulo $I_\Omega$ (which is
an ideal projector and can be computed efficiently for Gr\"obner and
H--bases), the
multiplication $p \mapsto \nu \left( (\cdot)_j p \right)$ is a linear
operation on $\Pi / I_\Omega$ for any $j=1,\dots,s$. It can be 
represented with respect to the basis $P$ by means of a matrix $M_j$
which is called $j$th \emph{multiplication table} and gives the
multivariate generalization of the \emph{Frobenius companion matrix}.

Due to the unique solvability of the Hermite interpolation problem in
$\Pi / I_\Omega$, there exists a basis of fundamental polynomials
$\ell_{\omega,q}$, $q \in Q_\omega$, $\omega \in \Omega$, such that
\begin{equation}
  \label{eq:FundamentalDef}
  q'(D) \ell_{\omega,q} (\xi_{\omega'}) = \delta_{\omega,\omega'}
  \delta_{q,q'}, \qquad q' \in Q_{\omega'}, \quad \omega' \in \Omega.
\end{equation}
The projection to the normal set, i.e., the interpolant, can now be
written for any $p \in \Pi$  as
$$
L p = \sum_{\omega \in \Omega} \sum_{q \in Q_\omega} q(D) p
(\xi_\omega) \, \ell_{\omega,q}
$$
hence, by the Leibniz rule and the fact that $\cQ_\omega$ is $D$--invariant
\begin{eqnarray}
  \nonumber
  L \left( (\cdot)_j \, \ell_{\omega,q} \right)
  & = & \sum_{\omega' \in \Omega}
        \sum_{q' \in Q_{\omega'}} q'(D) \left( (\cdot)_j \,
          \ell_{\omega,q} \right)
        (\xi_\omega) \, \ell_{\omega',q'} \\
  \nonumber
  & = & \sum_{\omega' \in \Omega}
        \sum_{q' \in Q_{\omega'}} \left( (\cdot)_j \, q'(D) \ell_{\omega,q}
        (\xi_\omega) + \frac{\partial q'}{\partial x_j} (D) \ell_{\omega,q}
        (\xi_\omega) \right) \, \ell_{\omega',q'} \\
  \nonumber
  & = & (\cdot)_j \, \ell_{\omega,q} + \sum_{q' \in Q_\omega}
        \sum_{q'' \in Q_\omega} c_j (q',q'') \, q''(D) \ell_{\omega,q}
        (\xi_\omega) \, \ell_{\omega,q'} \\
  \label{eq:InvariantSpace}
  & = & (\cdot)_j \, \ell_{\omega,q} + \sum_{q' \in Q_\omega} c_j
        (q,q') \, \ell_{\omega,q'},
\end{eqnarray}
where the coefficients $c_j (q,q')$ are defined by the expansion
\begin{equation}
  \label{eq:DerivExpansion}
  \frac{\partial q}{\partial x_j} = \sum_{q' \in Q_\omega} c_j (q',q) \,
  q', \qquad q \in Q_\Omega.  
\end{equation}
Note that the coefficients in (\ref{eq:DerivExpansion}) are zero if
$\deg q' \ge \deg q$. Therefore $c_j (q,q') = 0$ in
(\ref{eq:InvariantSpace}) if $\deg q' \ge \deg q$. In particular,
since $g_\omega$ is the unique element of maximal degree in
$\cQ_\omega$, it we have that
\begin{equation}
  \label{eq:MutpliTabEigen}
  L \left( (\cdot)_j \, \ell_{\omega,g_\omega} \right) = (\cdot)_j \,
  \ell_{\omega,g_\omega}, \qquad \omega \in \Omega.
\end{equation}
This way, we have given a short and simple proof of the
following result from \cite{MoellerStetter95}, restricted to our
special case of principal multiplicities.

\begin{theorem}
  The \emph{eigenvalues} of the multiplication tables $M_j$ are the
  components of the zeros $(\xi_\omega)_j$, $\omega \in \Omega$, the associated
  eigenvectors the polynomials $\ell_{\omega,g_\omega}$ and the other
  fundamental polynomials form an invariant space.
\end{theorem}

\noindent
In view of numerical linear algebra, the eigenvalue problems for
ideals with multiplicities become unpleasant as in general the
matrices become derogatory, except when $g_\omega$ is a power of
linear function, i.e., $g_\omega = ( v^T \cdot)^{\deg g_\omega}$ for
some $v \in \RR^s$, but the method by M\"oller and Tenberg
\cite{moeller01:_multiv} to determine the joint eigenvalues of
multiplication tables and their multiplicities also works in this situation.

There is another remedy described in
\cite[p.~48]{CohenCuypersSterk99}: building a matrix from traces of
certain multiplication tables, one can construct a basis for the
associated \emph{radical} ideal with simple zeros, thus avoiding the
hassle with the structure of multiplicities. In addition, this method
also gives the dimension of the multiplicity spaces which is
sufficient information to recover the polynomial coefficients.
Though this approach is surprisingly elementary, we
will not go into details here as it is not in the scope of the paper,
but refer once more to the recommendable collection
\cite{CohenCuypersSterk99}.

Moreover, the dimension of the respective invariant spaces is an upper
bound for $\deg f_\omega$ which can help to set up the parameters in
the interpolation problem in Section~\ref{sec:Prony}.

\section{Conclusion}
\label{sec:Conclusion}

The generalized version of Prony's problem with polynomial
coefficients is a straightforward extension of the standard problem
with constant coefficients. The main difference is that in
(\ref{eq:exppolyDef}) \emph{multiplicities} of common zeros in an
ideal play a role where the multiplicity spaces are related to the
shift invariant space generated by the coefficients via the operator
$L$ from (\ref{eq:LOperator}). This operator which relates the Taylor
expansion and interpolation at integer points in the Newton form, has
in turn a natural relationship with multivariate Stirling numbers of
the second kind. These properties can be used to extend the algorithms
from \cite{Sauer15S,Sauer2016A} almost without changes to the
generalized case, at least as far the construction of a good basis for
the Prony ideal is concerned.

The algorithms from \cite{Sauer15S,Sauer2016A}, numerical or symbolic,
can be reused, the only difference lies in multiplication tables with
multiplicities, but the tools from \cite{moeller01:_multiv} are also
available in this case and allow to detect zeros \emph{and} their
structure.

Implementations, numerical tests and comparison with the algorithms
from \cite{Mourrain16P} are straightforward lines of further work and
my be a worthwhile waste of time.


\end{document}